\documentclass[a4paper,11pt,reqno]{amsart}
\usepackage[centering, totalwidth = 390pt, totalheight = 610pt]{geometry}
\usepackage{amsmath,amssymb,amsfonts, url, stmaryrd}
\usepackage{eucal}
\usepackage{verbatim}
\usepackage{amsthm}
\usepackage{xspace}
\usepackage{hyperref}
 \usepackage{color}

\usepackage{enumerate}

\numberwithin{equation}{section}

\theoremstyle{plain}

\newtheorem{Theorem}{Theorem}

\newtheorem{Corollary}[Theorem]{Corollary}
\newtheorem{Proposition}[Theorem]{Proposition}
\newtheorem{Lemma}[Theorem]{Lemma}

\theoremstyle{definition}

\newtheorem{Conditions}[Theorem]{Conditions}

\newtheorem{example}[Theorem]{Example}

\newcommand{\Cat}{\ensuremath{\textnormal{Cat}}\xspace}

\newcommand{\Set}{\ensuremath{\textnormal{Set}}\xspace}
\newcommand{\SSet}{\ensuremath{\textnormal{SSet}}\xspace}

\newcommand{\f}[1]{\ensuremath{\mathcal{#1}}\xspace}
\newcommand{\g}[1]{\ensuremath{\mathbb{#1}}\xspace}

\newcommand{\TAlg}{\ensuremath{\textnormal{T-Alg}}\xspace}

\newcommand{\TAlgs}{\ensuremath{\textnormal{T-Alg}_{\textnormal{s}}}\xspace}

\newcommand{\RH}{\ensuremath{\textnormal{RH}}\xspace}

\newcommand{\Ord}{Ord}
\newcommand{\C}{{\mathcal C}}

\input xy
\xyoption{all}
\SelectTips{cm}{11}
\makeatletter
\def\matrixobject@{%
 \edef \next@{={\DirectionfromtheDirection@ }}%
 \expandafter \toks@ \next@ \plainxy@
 \let\xy@@ix@=\xyq@@toksix@
 \xyFN@ \OBJECT@}
\let\xy@entry@@norm=\entry@@norm
\def\entry@@norm@patched{%
 \let\object@=\matrixobject@
 \xy@entry@@norm }
\AtBeginDocument{\let\entry@@norm\entry@@norm@patched}
\makeatother

\newcommand{\twocong}[2][0.5]{\ar@{}[#2] \save ?(#1)*{\cong}\restore}
\newcommand{\twoeq}[2][0.5]{\ar@{}[#2] \save ?(#1)*{=}\restore}
\newcommand{\rtwocell}[3][0.5]{\ar@{}[#2] \ar@{=>}?(#1)+/l 0.2cm/;?(#1)+/r 0.2cm/^{#3}}
\newcommand{\ltwocell}[3][0.5]{\ar@{}[#2] \ar@{=>}?(#1)+/r 0.2cm/;?(#1)+/l 0.2cm/^{#3}}
\newcommand{\ltwocello}[3][0.5]{\ar@{}[#2] \ar@{=>}?(#1)+/r 0.2cm/;?(#1)+/l 0.2cm/_{#3}}
\newcommand{\dtwocell}[3][0.5]{\ar@{}[#2] \ar@{=>}?(#1)+/u  0.2cm/;?(#1)+/d 0.2cm/^{#3}}
\newcommand{\dltwocell}[3][0.5]{\ar@{}[#2] \ar@{=>}?(#1)+/ur  0.2cm/;?(#1)+/dl 0.2cm/^{#3}}
\newcommand{\drtwocell}[3][0.5]{\ar@{}[#2] \ar@{=>}?(#1)+/ul  0.2cm/;?(#1)+/dr 0.2cm/^{#3}}
\newcommand{\dthreecell}[3][0.5]{\ar@{}[#2] \ar@3{->}?(#1)+/u  0.2cm/;?(#1)+/d 0.2cm/^{#3}}
\newcommand{\utwocell}[3][0.5]{\ar@{}[#2] \ar@{=>}?(#1)+/d 0.2cm/;?(#1)+/u 0.2cm/_{#3}}
\newcommand{\dtwocelltarg}[3][0.5]{\ar@{}#2 \ar@{=>}?(#1)+/u  0.2cm/;?(#1)+/d 0.2cm/^{#3}}
\newcommand{\utwocelltarg}[3][0.5]{\ar@{}#2 \ar@{=>}?(#1)+/d  0.2cm/;?(#1)+/u 0.2cm/_{#3}}
\newcommand{\cd}[2][]{\vcenter{\hbox{\xymatrix#1{#2}}}}

\newcommand{\atwo}{\textbf{2}\xspace}
\newcommand{\Arr}{\textnormal{Arr}}

\newcommand{\Alg}{\ensuremath{\textnormal{Alg}}\xspace}

\newcommand{\Ainj}{\ensuremath{{\bbi}\textnormal{nj}\xspace}}
\newcommand{\AFib}{\ensuremath{\textnormal{AlgFib}}}
\newcommand{\Inj}{\textnormal{Inj}\xspace}
\newcommand{\ca}{\ensuremath{\mathcal A}\xspace}
\newcommand{\cb}{\ensuremath{\mathcal B}\xspace}
\newcommand{\cc}{\ensuremath{\mathcal C}\xspace}
\newcommand{\ce}{\ensuremath{\mathcal E}\xspace}

\newcommand{\ck}{\ensuremath{\mathcal K}\xspace}

\newcommand{\cm}{\ensuremath{\mathcal M}\xspace}

\newcommand{\cs}{\ensuremath{\mathcal S}\xspace}

\newcommand{\SE}{{\mathbb{S}\textnormal{E}}\xspace}

\newcommand{\bbi}{\ensuremath{\mathbb I}\xspace}

\newcommand{\WE}{\mathcal{W}}
\newcommand{\Fib}{\mathcal{F}}
\newcommand{\rlp}[1]{{#1}^\boxslash}
\newcommand{\llp}[1]{{}^\boxslash\!{#1}}

\def\black{\color{black}}

\title{Equipping weak equivalences with algebraic structure}
\author{John Bourke}
\address{Department of Mathematics and Statistics, Masaryk University, Kotl\'a\v rsk\'a 2, Brno 61137, Czech Republic}
\email{bourkej@math.muni.cz}
\begin{document}
\date{\today}


 
 \leftmargini=2em

\def\xypic{\hbox{\rm\Xy-pic}}
\maketitle

\begin{abstract}
We investigate the extent to which the weak equivalences in a model category can be equipped with algebraic structure.  We prove, for instance, that there exists a monad $T$ such that a morphism of topological spaces admits $T$-algebra structure if and only it is a weak homotopy equivalence.  Likewise for quasi-isomorphisms and many other examples.  The basic trick is to consider injectivity in arrow categories.  Using algebraic injectivity and cone injectivity we obtain general results about the extent to which the weak equivalences in a combinatorial model category can be equipped with algebraic structure.
\end{abstract}

\section{Introduction}

The notion of a Quillen model category \cite{Quillen1967Homotopical} plays a central role in modern homotopy theory.  It involves a pair of weak factorisation systems --- (trivial cofibration/fibration) and (cofibration/trivial fibration) --- and a class of morphisms called \emph{weak equivalences}.  In practice, most interesting model structures are \emph{cofibrantly generated}; that is, there exists a set $I \subseteq \Arr(\C)$ of trivial cofibrations such that $f:X \to Y$ is a fibration just when it has the right lifting property depicted below
\begin{equation*}
\cd{
A \ar[d]_{\alpha \in I} \ar[rr]^{r} && X \ar[d]^{f}\\
B \ar@{.>}[urr]^{\exists} \ar[rr]_{s} && Y}
\end{equation*}
and a set $J$ of cofibrations determining the trivial fibrations in the same way.  

The classical tool for constructing the weak factorisation systems (wfs) is a transfinite construction known as \emph{Quillen's small object argument}.  In \cite{Garner2011Understanding} Garner described a refined version of this construction, with better categorical properties, that produces not just weak factorisation systems but so-called \emph{algebraic weak factorisation systems} (awfs).  Notably the factorisations in such a system produce not merely fibrations but algebraic fibrations --- these are morphisms $f:X \to Y$ equipped with a lifting function $\phi$ providing solutions to each lifting problem as below
\begin{equation*}
\cd{
A \ar[d]_{\alpha \in I} \ar[rr]^{r} && X \ar[d]^{f}\\
B \ar@{.>}[urr]|{\phi(\alpha,r,s)} \ar[rr]_{s} && Y &.}
\end{equation*}
For instance, algebraising the usual Kan fibrations of simplicial sets leads to the \emph{algebraic Kan complexes} of \cite{Nikolaus2011Algebraic}, an algebraic model of $\omega$-groupoids.  The fact that algebraic fibrations come \emph{equipped} with liftings as opposed to having the property that certain liftings \emph{exist} gives one sense in which they are \emph{algebraic} --- another, more categorical formulation, is that they are the algebras for a monad on the category of arrows in $\C$.



Delving a little further into the theory one finds more powerful notions of cofibrant generation for awfs, in which the set $I$ is replaced by a category \cite{Garner2011Understanding} (or even a double category \cite{Bourke2016Accessible}) of morphisms.  These more highly structured notions encode, as algebraic fibrations, morphisms equipped with liftings satisfying \emph{coherence conditions}.  In 2-category theory, such coherence conditions enable us to capture lax structures such as Grothendieck fibrations and lax morphisms in two dimensional universal algebra \cite{Bourke2016Accessible, Bourke2016Weak} that lie far beyond the expressive power of wfs.  They also naturally arise in homotopy type theory as a means to construct good categorical models -- see, for instance, \cite{VG,Gambino2017The-Frobenius,Awodey2018ACubical}.



In \cite{Hess2017A-necessary} Emily Riehl and coauthors have introduced the notion of an \emph{accessible model category} --- a locally presentable model category whose two wfs underlie cofibrantly generated (aka accessible) awfs.  Since cofibrantly generated awfs are considerably more expressive than cofibrantly generated wfs, it follows that accessible model categories include all combinatorial model categories but form a broader class.  Examples of the former which are not combinatorial include the Hurewicz model structure on chain complexes \cite{Hess2017A-necessary} and the ``trivial" model structure on the arrow category of $\Cat$ \cite{Lack2007Homotopy-theoretic}.   

At CT2015 Riehl told me that she was hoping for a result that would do for accessible model categories what Smith's theorem \cite{Beke2000Sheafifiable} does for combinatorial ones.  Namely, Smith's theorem describes necessary and sufficient conditions on a set $I$ of generating cofibrations and class $W$ of weak equivalences to form part of a combinatorial model structure.  In the accessible setting, to capture examples such as the above two, one would need to allow at least a \emph{category} of generating cofibrations $I$.  The corresponding \emph{algebraic trivial fibrations} then come equipped with liftings satisfying compatibilities that make them more difficult to construct.  On contemplating this question, it quickly became clear to me that in order to construct algebraic trivial fibrations from morphisms which are both weak equivalences together and algebraic fibrations (whatever we take the latter to be) we would need to first understand the extent to which the weak equivalences \emph{can be made algebraic}.  

This brings us naturally to the question of what an \emph{algebraic model structure} ought to be.  The concept was first defined by Riehl \cite{Riehl2011Algebraic} as a model structure whose wfs underlie a pair of (inter-related) awfs.  Notably while the weak factorisation systems were algebraised, the class of weak equivalences was not.  In a development concurrent with our own and having motivations coming from homotopy type theory, Andrew Swan \cite{Swan2018Identity} has recently defined a stronger notion in which the class of weak equivalences is replaced by a category of structured equivalences.  
At this point in time, it is not yet clear whether a stable axiomatisation of the concept has been reached, and further examples are needed. But certainly the idea of strengthening Riehl's definition, by enhancing the class of weak equivalences to a category thereof, fits neatly with our own point of view.  What Swan's definition does not impose any conditions on is the sense in which these structured equivalences are algebraic.

%

The goal of the present paper is to investigate precisely this question in the classical setting of cofibrantly generated model categories: namely, to what extent can weak equivalences be made algebraic?  What does this mean?  Taking inspiration from the (trivial) fibrations in a cofibrantly generated model category we observe that they can be made algebraic because they can be captured using lifting properties.  However since classes of morphism defined by left/right lifting properties are stable under pushout/pullback respectively, and since interesting classes of weak equivalences are not stable under either construction, they cannot be defined using classical lifting properties.

Nonetheless it turns out that many kinds of weak equivalence can be characterised using lifting properties with respect to \emph{squares}.  Indeed, Jeff Smith, in a passing remark to the authors of \cite{Dugger2004Weak}, pointed out that a morphism $f:X \to Y$ of topological spaces is a weak homotopy equivalence just when the solid part of each commutative diagram as depicted below 
\begin{equation}\label{eq:arrowcat}
\cd{
S^{n-1} \ar[d]_{j_n} \ar@/^1.75pc/[rrrr]^{r} \ar[rr]^{j_n} && D^n \ar[d]^{\rho_n} \ar@{.>}[rr]^{\exists} && X \ar[d]^{f} \\
D^{n} \ar@/_1.75pc/[rrrr]_{s} \ar[rr]_{\tau_n} && D^{n+1} \ar@{.>}[rr]_{\exists} && Y
}
\end{equation}
can be extended to a commutative diagram as indicated.  (See Section~\ref{section:Examples} for more on this example and others like it.)  Surprisingly, seemingly nowhere has it been pointed out that these slightly odd looking lifting properties are instances of the categorical concept of injectivity --- namely, \emph{injectivity in the category of arrows}.  R\'{e}my Tuy\'{e}ras has noticed this independently.  

Taking injectivity in arrow categories as our starting point, in Section~\ref{section:Injectivity} we describe a plethora of classes of weak equivalences that naturally arise as injectives in arrow categories.  These include
\begin{itemize}
\item equivalences of categories;
\item weak homotopy equivalences of topological spaces;
\item quasi-isomorphisms of chain complexes;
\item weak equivalences of various kinds of (higher) categorical structures.
\end{itemize}
We then algebraise such weak equivalences as \emph{algebraic injectives}.  Doing so, it follows, perhaps surprising, that for each of the above classes
\begin{itemize}
\item there exists a monad $T$ on the arrow category such that a morphism bears $T$-algebra structure just when it is a weak equivalence.
\end{itemize}  
This is Corollary~\ref{cor:monad1} and provides a categorical sense in which such classes of weak equivalence can be made algebraic.  Our first main result, Theorem~\ref{thm:CombinatorialInj}, characterises those combinatorial model categories whose weak equivalences can be made algebraic, in either of the above senses, as those whose weak equivalences are stable under infinite products.  

Since weak equivalences of simplicial sets are not stable under infinite products they therefore cannot be captured using injectivity or monads.  Section~\ref{section:Cones} deals with the appropriate generalisations of these concepts --- cone injectivity and multimonads --- required to capture weak equivalences in a general combinatorial model category.  We introduce \emph{algebraic cone injectives} and show that, in a locally presentable category, the category of algebraic cone injectives is locally \emph{multipresentable} and \emph{multimonadic.}  This allows us to give our second main result, Theorem~\ref{thm:CombCone}, which describes the algebraic structure borne by weak equivalences in a general combinatorial model category.  Furthermore, using simplicial subdivision, we describe a concrete set of cones generating the weak equivalences of simplicial sets.

Finally, in Section~\ref{section:Nikolaus}, we sharpen a result of Nikolaus by showing that the category of algebraically fibrant objects in a combinatorial model category admits a transferred model structure Quillen equivalent to the original one --- in particular, this establishes the apparently new result that each combinatorial model category is equivalent to one in which \emph{all objects are fibrant}.  By passing to the category of algebraically fibrant objects this allows us to return from the non-standard world of cones and multimonads to the more familiar world of injectivity and monads. 

In future work we plan to build upon the results developed herein to obtain a version of Smith's theorem for accessible model categories.  At the same time, we plan to use Theorem~\ref{thm:CombCone} to obtain new techniques for constructing, and a deeper understanding of, algebraic model categories in the stronger sense of \cite{Swan2018Identity}.

\subsection*{Acknowledgements}

The author gratefully acknowledges the support of an Australian Research Council Discovery Grant DP160101519 and the support of the Grant Agency of the Czech Republic under the grant 19-00902S.  Particular thanks are due to Emily Riehl whose interest in an algebraic version of Smith's theorem got me thinking about this topic and to Luk\'{a}\v{s} Vok\v{r}\'{i}nek who helped me to see the connection between $Ex_{\infty}$ and the generating cones for simplicial sets.  Thanks also to the organisers of the PSSL101 in Leeds for providing the opportunity to present this work, and to the members of the Australian Category Seminar for listening to me speak about it.

\section{Injectivity in arrow categories and weak equivalences as algebras for a monad}\label{section:Injectivity}

\subsection{Injectivity in arrow categories}\label{sect:Idea}

Given morphisms $f:A \to B$ and $g:C \to D$ one says that $g$ has the \emph{right lifting property} (r.l.p.) with respect to $f$ if in each commutative square 
$$\cd{
A \ar[d]_{f} \ar[r]^{r} & C \ar[d]^{g}\\
B \ar@{.>}[ur]^{\exists} \ar[r]_{s} & D}$$
there exists a diagonal filler making both triangles commute.  We denote the relationship by $f \boxslash g$.  More generally if $J$ and $K$ are classes of morphisms we say that $J \boxslash K$ if $j \boxslash k$ for each $j \in J$ and $k \in K$.  We define $\rlp{J}  = \{f: J \boxslash f\}$ and $\llp{J}=\{f:f \boxslash J\}$.

Many properties of morphisms are lifting properties.  For instance, a functor $f:A \to B$ has the right lifting property with respect to 
\begin{enumerate}
\item $\varnothing \to \{\bullet\}$ just when it is surjective on objects;
\item $\{0 \hspace{0.5cm} 1\} \to \{0 \to 1\}$ just when it is full;
\item $\{0 \rightrightarrows 1\} \to \{0 \to 1\}$ just when it is faithful.
\end{enumerate}
Accordingly the surjective on objects equivalences of categories are of the form $J^{\boxslash}$ where $J$ consists of the above three morphisms in $\Cat$.  More generally, in a cofibrantly generated model category both the classes of fibrations and trivial fibrations are of the form $J^{\boxslash}$ for a set of morphisms $J$.

What about equivalences of categories?  The issue here is that being \emph{essentially} surjective on objects is not a lifting property.  We can, however, capture it in a similar fashion.  Consider the generic isomorphism $\{0 \sim 1\}$ and the two functors $0,1:\{\bullet\} \rightrightarrows \{0 \sim 1\}$ named by the objects they select.  We obtain a commutative square as in the inside left below.
\begin{equation}\label{eq:eso}
\cd{
\varnothing \ar[d]_{!} \ar@/^1.75pc/[rrrr]^{!} \ar[rr]^{!} && \{\bullet\} \ar[d]^{0} \ar@{.>}[rr]^{\exists} && A \ar[d]^{f} \\
\{\bullet\} \ar@/_1.75pc/[rrrr]_{b} \ar[rr]^{1} && \{0 \sim 1\} \ar@{.>}[rr]^{\exists } && B
}
\end{equation}
A commutative square from the left vertical morphism to $f$ specifies an object $b \in B$.  Given such, dotted arrows rendering the diagram everywhere commutative amount to the choice of an object $a \in A$ and an isomorphism $fa \sim b$.  Accordingly such dotted arrows provide witnesses to the essential surjectivity of $f$.

To gain a better grasp of the above condition we take a step backwards.   Given a morphism $f:A \to B$ and object $C \in \f C$ we say that $C$ is injective to $f$ 
$$\cd{
A \ar[d]_{f} \ar[r]^{r} & C \\
B \ar@{.>}[ur]_{\exists}}$$
if each $r:A \to C$ can be extended along $f$ as depicted.  We write $f \perp C$ and $J \perp C$ for the evident extension of this concept to deal with a class of morphisms $J$ and define $\Inj(J)=\{C:J \perp C\}$.

The relevant concept here is that of injectivity in the arrow category $\Arr(\cc)$.  Objects of $\Arr(\cc)$ are morphisms in $\cc$ whilst a morphism $(r,s):f \to g \in \Arr(\cc)$ is a commutative square as on the inside left below.  To say that $(r,s) \perp h$ is then to say that given the solid part of a diagram as below
\begin{equation}\label{eq:arrowcat}
\cd{
A \ar[d]_{f} \ar@/^1.75pc/[rrrr]^{t} \ar[rr]^{r} && C \ar[d]^{g} \ar@{.>}[rr]^{\exists v} && E \ar[d]^{h} \\
B \ar@/_1.75pc/[rrrr]_{u} \ar[rr]_{s} && D \ar@{.>}[rr]_{\exists w} && F
}
\end{equation}
there exist morphisms $v$ and $w$ as depicted, such that the right square commutes and such that $v \circ r = t$ and $w \circ s = u$.

Since \eqref{eq:eso} is an instance of \eqref{eq:arrowcat} we conclude that essentially surjective on objects functors are injectives in $\Arr(\Cat)$.  Now if the ambient category $\f C$ admits a terminal object $1$ then $f \perp C$ just when $!_{C}:C \to 1$ has the r.l.p. with respect to $f$.  Therefore injectivity is, ordinarily, a special case of having the right lifting property.  But, in fact, having the right lifting property is \emph{always} a special case of injectivity --- once we pass to the arrow category.  To see this observe that given $f:A \to B$ in $\cc$ we obtain the morphism $(f,1_{B}):f \to 1_{B}$ in $\Arr(\cc)$ as below.
$$
\cd{A \ar[d]_{f} \ar[r]^{f} & B \ar[d]^{1_{B}} \\
B \ar[r]_{1_{B}} & B \rlap{ .}
}$$
in $\Arr(\cc)$.  A moment's thought establishes the following result.
\begin{Lemma}
Consider morphisms $f:A \to B$ and $g:C \to D$ in $\f C$.  Then $f \boxslash g$ if and only if $(f,1_{B}) \perp g$ in $\Arr(\cc)$. 
\end{Lemma}

Putting the above together we conclude that equivalences of categories are the injectives with respect to the three morphisms in $\Arr(\Cat)$ below.
\begin{equation}\label{eq:CatInjectives}
\cd{
\varnothing \ar[d]_{!} \ar[r]^-{!} & \{\bullet\} \ar[d]^{0} & \{0 \hspace{0.5cm} 1\}\ar[d] \ar[r] & \{0 \to 1\}\ar[d] & \{0 \rightrightarrows 1\}\ar[d] \ar[r] & \{0 \to 1\}\ar[d] \\
\{\bullet\} \ar[r]^-{1} & \{0 \sim 1\} & \{0 \to 1\} \ar[r] & \{0 \to 1\} & \{0 \to 1\} \ar[r] & \{0 \to 1\}
}
\end{equation}
We observe in passing that each is a morphism from a (generating) cofibration to a trivial cofibration in the categorical model structure on $\Cat$ --- a pattern that repeats itself in many of the examples detailed below.

\subsection{Examples}\label{section:Examples}
Here are further classes of model categories whose weak equivalences can be described as injectives in the arrow category.

\begin{example}[Topological spaces]\label{example:topological}
Let $D^n$ denote the unit disk in $\g R^{n}$ and $$\rho_{n},\tau_{n}:D^n \rightrightarrows D^{n+1}$$ denote the inclusions to the north and south hemisphere.  Let $S^{n-1}$ denote the unit sphere in $\g R^{n}$ and $j_{n}:S^{n-1} \to D^{n}$ be the inclusion of the boundary.  Note we take $D^{0}=1$ and $S^{-1}=\varnothing$.  One then has a commuting square
$$
\cd{S^{n-1} \ar[d]_{j_{n}} \ar[r]^{j_{n}} & D^{n} \ar[d]^{\rho_{n}} \\
D^{n} \ar[r]_{\tau_{n}} & D^{n+1}}
$$
for each $n \in \g N$.  As noted in the introduction to \cite{Dugger2004Weak} a continuous map $f:X \to Y$ is a weak homotopy equivalence of topological spaces just when it is injective to the above squares.

Moreover, we note that the left and right vertical maps in the square form the generating cofibrations and trivial cofibrations for the classical model structure on $Top$ whose weak equivalences are the weak homotopy equivalences \cite{Quillen1967Homotopical}.  In particular, the above squares concisely encode the standard model category structure on topological spaces.
\end{example}

\begin{example}[Chain complexes]
Let $Ch_{R}$ denote the category of unbounded chain complexes over a commutative ring.  In a manner similar to the topological example above, we will describe the \emph{quasi-isomorphisms} as an injectivity class.  
\begin{Lemma}\label{thm:quasi}
A morphism $f:X \to Y$ is a quasi-isomorphism if and only if the following condition holds for each $n \in \g Z$.
\begin{itemize}
\item Given an $n$-cycle $a \in X_{n}$ and element $b \in Y_{n+1}$ with $db=fa$, there exists $c \in X_{n+1}$ such that $dc=a$ \textbf{and $e \in Y_{n+2}$ such that $de=fc-b$.}
\end{itemize}
\end{Lemma}
\begin{proof}
To say that $H_{n}f:H_{n}X \to H_{n}Y$ is monic is clearly to say exactly that the part of the above condition that is not in bold holds.

Suppose the full condition holds --- we must show that $H_{n}f$ is surjective.  For this let $u \in Y_{n}$ be an $n$-cycle.  Take $a=0$ and $b=u$; then $db = 0 = a$ so there exists $c \in X_{n}$ satisfying $dc=0$ and $e \in Y_{n+1}$ such that $de=fc-u$. Hence $[u]=[fc] \in H_{n}Y$ as required.

Conversely suppose that $H_nf$ is a quasi-isomorphism.  Let $a \in X_{n}$ and $b \in Y_{n+1}$ be as above.  As $H_{n}f$ is monic there exists $c \in X_{n+1}$ such that $dc=a$.  Now $fc-b \in Y_{n+1}$ is a cycle since $d(fc-b)=fdc-db=fa-fa=0$.  So by surjectivity there exists a cycle $e \in X_{n+1}$ and $h \in Y_{n+2}$ such that $dh=(fc-b)-fe$.  Since $dh=f(c-e)-b$ and $d(c-e)=dc-de=a-0=a$ the pair $(c-e,h)$ verify the full condition.
\end{proof}
We use the standard topological names for the following chain complexes. 
$$(S^n)_k = \begin{cases} R & k = n \\ 0 & k \neq n \end{cases} \quad (D^n)_k = \begin{cases} R & k = n,n-1 \\ 0 & \mathrm{else} \end{cases} \quad (I^{n})_k = \begin{cases} R & k=n+1 \\ R \oplus R & k = n \\ 0 & \mathrm{else} \end{cases}$$

The non-trivial differential in $D^n$ is the identity; the non-trivial differential in $I^{n}$ is $x \mapsto (x,-x)$.  There is an evident inclusion $j_n:S^{n-1} \to D^{n}$ and a pair of inclusions $i,j:D^{n} \rightrightarrows I^{n}$ which act as the coproduct inclusions $R \rightrightarrows R \oplus R$ in degree $n$ and as zero otherwise.  We obtain a commutative square 
$$
\cd{ S^{n-1} \ar[d]_{j_{n}} \ar[r]^{j_{n}} & D^{n} \ar[d]^{\rho_n} \\
D^{n} \ar[r]_{\tau_n} & I^{n}}$$
against which $f:X \to Y$ is injective for each $n$ exactly when it verifies the criterion for a quasi-isomorphism given in Lemma~\ref{thm:quasi}.

The connection between the above squares and those of Example~\ref{example:topological} can be sharpened
by noting the following connection between their bottom right corners --- namely that, in the topological setting $D^{n+1} \cong D^{n} \times I$ for $I$ the unit interval, whilst in the setting of chain complexes $I^n \cong  D^n \otimes I$ for $I = I^1$ defined as above.
\end{example}

\begin{example}[Categories with structure]
Categories equipped with structure --- such as monoidal categories --- can typically be understood as the algebras for an accessible 2-monad $T$ on $\Cat$.  For such a $T$ the category $\TAlgs$ of algebras and strict morphisms admits a model structure \cite{Lack2007Homotopy-theoretic} in which $f:A \to B$ is a weak equivalence just when its image under the forgetful functor $U:\TAlgs \to \Cat$ is an equivalence of categories.  Since we have an adjunction $F \dashv U$ it follows that an algebra map $f:A \to B$ is a weak equivalence just when it is injective with respect to the image of the three squares \eqref{eq:CatInjectives} under $\Arr(F):\Arr(\Cat) \to \Arr(\TAlgs)$.
\end{example}

\begin{example}[2-categories, bicategories and Gray-categories]
There are model structures on the categories of 2-categories \cite{Lack2002A-quillen}, of bicategories \cite{Lack2004A-quillen}, and of Gray-categories \cite{Lack2011A-quillen} established by Lack.  In each case the weak equivalences --- biequivalences in the first two cases and the triequivalences in the third --- can be described using injectivity conditons, much as for equivalences of categories.
\end{example}

\begin{example}[Globular $\omega$-groupoids and $\omega$-categories]
There are various ways of describing the weak equivalences of strict $\omega$-groupoids.  Firstly there is what we might dub the topological definition which is given in terms of homotopy groups.  There is also a categorical definition which is as follows.  A morphism $f:X \to Y$ is a weak equivalence if
\begin{enumerate}
\item given $y \in Y(0)$ there exists $x \in X(0)$ and a 1-cell $\alpha:Fx \cong y$, and
\item for $n \geq 0$ if $x,y \in X(n)$ are parallel $n$-cells then given $\alpha:fx \to fy \in Y(n+1)$ there exists $\beta:x \to y \in X(n+1)$ and $\rho:F\beta \cong \alpha \in Y(n+2)$.
\end{enumerate}
(Note that here our convention is that all $0$-cells are parallel.)  The content of Proposition 1.7(iv) of \cite{Ara2011The-Brown} is that the categorical and topological definitions coincide.

For all $n$ let $D^{n}=\mathbb G(-,n) \in [\mathbb G^{op},\Set]$ be the $n$-globe and $j_{n}:S^{n-1} \hookrightarrow D^{n}$ be the globular set obtained by omitting the unique cell of dimension $n$.  We have a commutative square
$$
\cd{S^{n-1} \ar[d]_{j_{n}} \ar[r]^{j_{n}} & D^{n} \ar[d]^{\rho_{n}} \\
D^{n} \ar[r]_{\tau_{n}} & D^{n+1}}
$$
where $\rho_{n}$ and $\tau_{n}$ are the evident inclusions.  Let $U:\omega\textnormal{-Gpd} \to [\mathbb G^{op},\Set]$ denote the forgetful functor and $F$ its left adjoint.
Then $f:X \to Y$ is a weak equivalence just when $(j_{n},\tau_n) \perp Uf$ or equivalently $(Fj_{n},F\tau_n) \perp f$ for all $n$.  This time the sets $\{Fj_{n}:n \in \g N\}$ and $\{F\rho_{n}:n \in \g N\}$ are the generating cofibrations and trivial cofibrations for the Brown-Gola\'{n}ski model structure on strict $\omega$-groupoids.

An essentially identical injectivity characterisation is possible for weak equivalences of Grothendieck weak $\omega$-groupoids.  This follows from Theorem 4.18(iv) of \cite{Ara2013On-the}.  It is expected, though not yet proven, that these are the weak equivalences of a model structure.  A slightly more complex injectivity characterisation can be given for weak equivalences of strict $\omega$-categories.  This is the content of Proposition 4.37 of \cite{Lafont2010A-folk}.
\end{example}

\begin{example}[{Pure monomorphisms}]
Here is an example from outside of homotopy theory.  In a locally presentable category $\f C$ a morphism $f:A \to B$ is said to be a \emph{pure monomorphism} if in each square
\begin{equation*}
\cd{
n \ar[d]_{j} \ar[r]^{r} & A \ar[d]^{f} &&& n \ar[d]_{j} \ar[r]^{r} & A \\
m \ar[r]_{s} & B &&& m \ar@{..>}[ur]_{\exists}
}
\end{equation*}
with $n$ and $m$ finitely presentable objects there exists a diagonal $m \to A$ such that the triangle above right commutes.

This condition is equivalent to asking that $f:A \to B$ be injective in $\Arr(\cc)$ with respect to the pushout square
\begin{equation*}
\cd{n \ar[d]_{j} \ar[r]^{j} & m \ar[d]\\
m \ar[r] & m \cup_{n} m .} 
\end{equation*}
Taking such a pushout square for each morphism $f:n \to m$ between finitely presentable objects (of which there is only a set up to isomorphism) we obtain the pure monos as the corresponding injectives in $\Arr(\cc)$.
\end{example}

\subsection{Algebraic injectives and weak equivalences as the algebras for a monad}\label{section:AInj}
Let $J$ be a class of morphisms in $\f C$.  We often identify the class $\Inj(J)$ of $J$-injective objects with the corresponding full subcategory of $\f C$.

We enhance this by considering the category $\Ainj(J)$ of \emph{algebraic injectives}, an object of which is given by a pair $(C,c)$ where $C \in \cc$ together with extensions
$$\cd{A \ar[d]_{j} \ar[r]^{f} & C\\
B \ar@{.>}[ur]_{c(j,f)}}$$
for each lifting problem.  Morphisms $f:(C,c) \to (D,d)$ are morphisms of $\cc$ commuting with the given extensions.  

Observe that whilst an object $C$ is injective just when for all $j:A \to B \in J$ the function 
$$\f C(j,C):\f C(B,C) \to \f C(A,C)$$
is surjective, the algebraic variant enhances this by specifying a choice of section $c(j,-)$ for each such function.  Under this viewpoint, the morphisms of algebraic injectives are those commuting with the sections.  

The above can be concisely encoded by the fact that the square
\begin{equation}\label{eq:pullback1}
\cd{\Ainj(J) \ar[d]_{U} \ar[r]^{} & \SE([J,\Set]) \ar[d]^{V} \\
\f C \ar[r]^-{K} & \Arr([J,\Set])}
\end{equation}
is a pullback.  Here $\SE([J,\Set])$ is the category of \emph{split epimorphisms} in $[J,\Set]$ and $K$ the functor sending $C$ to the family $(\f C(j,C):C(B,C) \to \f C(A,C))_{j \in J}$.

\begin{example}[Equivalences of categories algebraically]
If $\cc=\Cat^{2}$ and $J$ consists of the three morphisms of \eqref{eq:CatInjectives} then an object of $\Ainj(J)$ is given by a fully faithful functor $f:A \to B$ together with, for each $b \in B$, an object $a_{b} \in A$ and choice of isomorphism $\phi_b:fa_{b} \cong b$.  That is, an equivalence of categories equipped with suitable witnesses to its essential surjectivity.  The morphisms are commutative squares preserving the chosen witnesses.
\end{example}

The category of algebraic injectives comes equipped with a forgetful functor $U:\Ainj(J) \to \cc$.  It follows from the work of Garner \cite{Garner2011Understanding} on algebraic weak factorisation systems that this forgetful functor, under rather general assumptions covering all of the examples thus far, has a left adjoint and is \emph{strictly monadic}.  Given that we are not interested in algebraic weak factorisation systems but merely algebraic injectives, we can explain this without too much trouble --- as we now do.

If $\cc$ is cocomplete we can, for each $C \in \cc$, form the pushout $RC$
\begin{equation}\label{eq:rconstruction}
\cd{
\Sigma_{j:A \to B \in J}\cc(A,C).A \ar[rr]^-{\epsilon_{C}} \ar[d]_{\Sigma_{j:A \to B \in J}1.\alpha} && C \ar[d]^{\eta_{C}}\\
\Sigma_{j:A \to B \in J}\cc(A,C).B \ar@{.>}[urr]^{c} \ar[rr] && RC
}
\end{equation}

in which $\epsilon_{C}$ is the unique map corresponding to the function $$\Sigma_{j:A \to B \in J}\cc(A,C) \to \cc(A,C):(j,f )\mapsto f.$$  Then $(R,\eta)$ is a pointed endofunctor and the universal property of the pushout $RC$ ensures that an $(R,\eta)$-algebra structure on $C$ amounts to a morphism $c$ as above rendering commutative the upper left triangle.  This in turn amounts to giving a section $c(j,-)$ of the function
$$\cc(j,C):\cc(B,C) \to \cc(A,C)$$
for each $j \in J$  --- that is, to the structure of an algebraic injective.  In this way we obtain an isomorphism $$\Ainj(J) \cong (R,\eta)\textnormal{-Alg}$$ over $\cc$.  Consequently free algebraic injectives exist just when free algebras for the pointed endofunctor $(R,\eta)$ do. In Appendix~\ref{sect:Pointed} we give a thorough treatment of the construction of free algebras for pointed endofunctors --- and so of free algebraic injectives --- but in the present section we content ourselves with citing existence results from the literature. Two size conditions guaranteeing existence --- identified and discussed further in Section 4 of \cite{Garner2011Understanding} --- are the following.

\begin{Conditions}\label{thm:size}
\begin{enumerate}
\item For each $X \in \cc$ there exists a regular cardinal $\alpha_X$ such that $\cc(X,-)$ preserves $\alpha_X$-filtered colimits.
\item $\cc$ admits a proper well copowered factorisation system $(\ce,\cm)$ and for each $X \in \cc$ there exists a regular cardinal $\alpha_X$ such that $\cc(X,-)$ preserves $\alpha_X$-filtered unions of $\cm$-subobjects.
\end{enumerate}
\end{Conditions}

(1) is stronger than (2) --- take the (Iso,All)-factorisation system --- and is satisfied by any locally presentable category.  This covers all of the examples of Section~\ref{section:Examples} except for topological spaces.  This last category does, however, satisfy (2) on taking $\ce$ to be the class of surjections and $\cm$ the class of subspace embeddings.


\begin{Theorem}\label{thm:AlgebraicInjectives}
Let $J$ be a set of morphisms and $\cc$ a cocomplete category satisfying either of the conditions in Conditions~\ref{thm:size}.  
\begin{enumerate}
\item Then the forgetful functor $U:\Ainj(J) \to \cc$ has a left adjoint and is strictly monadic.  
\item If moreover $\cc$ is locally presentable then $\Ainj(J)$ is too and $U$ is accessible.  
\end{enumerate}
\end{Theorem}

\begin{proof}
Since $U$ is, up to isomorphism over $\cc$, the forgetful functor from the category of $R$-algebras it creates $U$-absolute coequalisers.  Therefore it is strictly monadic if it has a left adjoint.  

Now suppose that the size condition \ref{thm:size}.2 holds.  Then there exists a $\lambda$ such that $\cc(X,-):\cc \to \Set$ preserves $\lambda$-filtered unions of $\cm$-subobjects where $X$ is the source of any morphism appearing in $J$.  It follows that the two endofunctors $C \mapsto \Sigma_{j:A \to B \in J}\cc(A,C).A$ and $C \mapsto\Sigma_{j:A \to B \in J}\cc(A,C).B$ have the same preservation property whence so does the pushout $R$.  Hence by Theorems' 14.3 and 15.6 of \cite{Kelly1980A-unified} the free $(R,\eta)$-algebra exists.  Alternatively see Appendix~\ref{sect:Pointed}.  Since Condition \ref{thm:size}.2 implies \ref{thm:size}.1 we have proven the first part.

Finally suppose that $\cc$ is locally presentable.  Then \ref{thm:size}.1 holds and $R$ preserves $\lambda$-filtered colimits for $\lambda$ constructed as above.  Hence $U:(R,\eta)\textnormal{-Alg} \to \cc$ creates them whence the induced monad $T=UF$ preserves them too.  Since, by \cite[Satz~10.3]{Gabriel1971Lokal}, the category of algebras for a $\lambda$-accessible monad on a locally $\lambda$-presentable category is again locally $\lambda$-presentable we are done.
\end{proof}

\begin{Corollary}\label{cor:monad1}
There is a monad $T$ on $\Arr(Top)$ such that a morphism bears $T$-algebra structure just when it is a weak homotopy equivalence.  Likewise there are monads detecting quasi-isomorphisms of chain complexes, equivalences of categories and of higher categories, pure monomorphisms and all of the examples from Section~\ref{section:Examples}.
\end{Corollary}
\begin{proof}
For each of these categories $\cc$ and class of morphisms $W$ we have, in Section~\ref{section:Examples}, described a set of morphisms $J$ of $\Arr(\cc)$ with $W=\Inj(J)$.  Now $f \in \Arr(\cc)$ belongs to $\Inj(J)$ if and only if it can be equipped with the structure of an algebraic injective $(f,\phi) \in \Ainj(J)$.\begin{footnote}{Of course this assertion makes use of the axiom of the choice.  Indeed when $J$ consists of the single morphism  
$$
\cd{\varnothing \ar[d]_{} \ar[r]^{} & 1 \ar[d]^{} \\
1 \ar[r] & 1
}$$
in $\Arr(\Set)$ it \emph{is} the axiom of choice!}\end{footnote}  By Theorem~\ref{thm:AlgebraicInjectives} the forgetful functor $U:\Ainj(J) \to \Arr(\cc)$ has a left adjoint and is strictly monadic.  Writing $T=UF$ for the monad induced by the adjunction it follows that $f$ admits the structure of an algebraic injective if and only if it admits the structure of a $T$-algebra.
\end{proof}



\subsection{Injectivity in locally presentable categories and weak equivalences in combinatorial model categories}\label{section:InjComb}

In Section~\ref{section:Examples} we have seen that the weak equivalences in many Quillen model categories can be described as injectives in the arrow category.  In which Quillen model categories $\cc$ is this the case?  

In the present section we will give a complete answer to this question in the case of \emph{combinatorial} model categories.  Recall that a model category $\cc$ is said to be combinatorial if it is both locally presentable and cofibrantly generated.  Our result follows easily from the following result, Theorem 4.8 of \cite{Adamek1994Locally}.

\begin{Theorem}\label{thm:Injectives}\textnormal{(Ad{\'a}mek and Rosick{\'y})}
Let $\cc$ be locally presentable.  A full subcategory $j:\ca \hookrightarrow \cc$ is of the form $\Inj(J)$ for $J$ a set of morphisms if and only if $\ca$ is accessible, accessibly embedded and closed under products in $\cc$.
\end{Theorem}

The proof in \cite{Adamek1994Locally} uses the fact that each injectivity class admits a full embedding into the category of graphs.  This seems to the author rather ad-hoc.  As an application of \emph{algebraic injectivity}, we give a novel proof that avoids any such embedding.

As in \cite{Adamek1994Locally} we will use the \emph{uniformization theorem} for accessible categories --- see Theorem 2.19 of \emph{ibid.} or Theorem 2.49 of \cite{Makkai1989Accessible} for the original reference --- the relevant part of which asserts the following.

\begin{itemize}
\item 
Let $U:\ca \to \cb$ be an accessible functor between accessible categories.  There exist arbitarily large regular cardinals $\lambda$ for which $\ca$ and $\cb$ are $\lambda$-accessible and such that $U$ preserves both $\lambda$-presentable objects and $\lambda$-filtered colimits.
\end{itemize}

\begin{proof}[Proof of Theorem \ref{thm:Injectives}]
Let $J$ and $\cc$ be given, and let $\lambda$ be such that each $j \in J$ has $\lambda$-presentable source and target.    Then the full subcategory $j:\Inj(J) \to \cc$ is closed under products and $\lambda$-filtered colimits.  

By Theorem~\ref{thm:AlgebraicInjectives} the category $\Ainj(J)$ of algebraic injectives is locally $\lambda$-presentable and $U:\Ainj(J) \to \cc$ a $\lambda$-accessible right adjoint.  Therefore by the uniformization theorem there exists $\mu \geq \lambda$ such that $U$ preserves $\mu$-filtered colimits and $\mu$-presentable objects and is a functor between locally $\mu$-presentable categories.

We claim that $\Inj(J)$ is $\mu$-accessible and $\mu$-accessibly embedded in $\cc$.  Since $j:\Inj(J) \to \cc$ preserves $\lambda$-filtered colimits it preserves $\mu$-filtered colimits. In particular, this implies that if $jX$ is $\mu$-presentable in $\cc$ then $X$ is $\mu$-presentable in $\Inj(J)$.

For $\mu$-accessibility we must exhibit a set $S$ of $\mu$-presentables in $\Inj(J)$ such that each $X \in \Inj(J)$ is a $\mu$-filtered colimit of those in $S$.  Since $\Ainj(J)$ is locally $\mu$-presentable it admits such a set of $\mu$-presentables and we define $S$ to consist of the image of these under the $\mu$-presentable preserving $U:\Ainj(J) \to \cc$.   Now $X \in \Inj(J)$ underlies $(X,x) \in \Ainj(J)$.  Consider the $\mu$-filtered colimit $(X,x)=col_{i \in I}(X_{i},x_{i})$ of $\mu$-presentables in $\Ainj(J)$.  By the above properties of $U$ the colimit $X=Ucol_{i \in I}(X_{i},x_{i})=col_{i \in I}X_{i} \in \cc$ is a $\mu$-filtered colimit of $\mu$-presentable objects in $\cc$ which are $J$-injective.  Since $j:\Inj(J) \to \cc$ is closed under $\mu$-filtered colimits these objects are also $\mu$-presentable in $\Inj(J)$ and so exhibit $Y$ as a $\mu$-filtered colimit in $\Inj(J)$ of objects in $S$.

The converse direction is exactly as in \cite{Adamek1994Locally} and we include it only for completeness.  Let $j:\ca \to \cc$ satisfy the stated properties.  By the uniformization theorem there exists $\mu$ such that \ca,\cc are $\mu$-accessible and such that $j$ preserves both $\mu$-presentable objects and $\mu$-filtered colimits.

Since $j$ is accessible it satisfies the solution set condition.  In particular, for each $\mu$-presentable object $X \in \cc$ the category $X/j$ admits a weakly initial set of objects.  Since $j$ preserves products it follows that $X/j$ admits them --- constructed as in \ca --- so that the product in $X/j$ of the weakly initial set of objects exists and forms a weakly initial object $X \to jY$.  Now write $Y$ as a $\mu$-filtered colimit of $\mu$-presentables $Y_{i}$.  Then $jY$ is still a $\mu$-filtered colimit, still of $\mu$-presentables $jY_{i}$, whence the morphism $X \to jY$ factors as $X \to jY_{i}$ for some $i$: itself now a weakly initial object in $X/j$.  In summary, for $\mu$-presentable $X$ the comma category $X/j$ has a weakly initial object $p_{X}:X \to jX^{\star}$ with $\mu$-presentable codomain. 

We take the union of these morphisms $$J=\{X \to j X^{\star}:X\textnormal{ } \mu\textnormal{-presentable} \}$$ over a representative set of the $\mu$-presentable objects in $\cc$ and claim that $\ca = \Inj(J)$. Weak initiality ensures that  each object of $\ca$ is $J$-injective.  For the reverse inclusion, let $Y$ be $J$-injective consider its canonical presentation as a $\mu$-filtered colimit of $\mu$-presentables.  Here the canonical diagram $Pres_{\mu}(Y)$ has for objects those morphisms $A \to Y$ with $A$ $\mu$-presentable.  The trick is to consider the full subcategory $\ck \hookrightarrow Pres_{\mu}(Y)$ consisting of those morphisms of the form $j X^{\star} \to Y$ and to show that the inclusion is cofinal: then $\ck$ will itself be $\mu$-filtered and, by cofinality, $Y$ a $\mu$-filtered colimit of objects in the image of $j$; since $\ca$ is closed under $\mu$-filtered colimits $Y$ will be in the image of $j:\ca \hookrightarrow \cc$ too.  Cofinality follows from the fact that since $Y \in \Inj(J)$ each morphism $A \to Y$ factors through $A \to jA^{\star}$  and that $Pres_{\mu}(Y)$ is filtered.
\end{proof}

The following result covers all of the model categorical examples of Section~\ref{section:Examples} except for the model category of topological spaces.

\begin{Theorem}\label{thm:CombinatorialInj}
Let $\cc$ be a combinatorial model category with class of weak equivalence $W$.  The following are equivalent.
\begin{enumerate}
\item $W \hookrightarrow \Arr(\cc)$ is closed under all small products.
\item $W \hookrightarrow \Arr(\cc)$ is of the form $\Inj(J)$ for $J$ a set of morphisms in $\Arr(\cc)$.
\item There exists a monad $T$ on $\Arr(\cc)$ such that a morphism $f$ bears $T$-algebra structure just when it belongs to $W$.
\item There exists an accessible monad $T$ on $\Arr(\cc)$ such that a morphism $f$ bears $T$-algebra structure just when it belongs to $W$.
\end{enumerate}
In particular, these equivalent conditions hold whenever all objects in $\cc$ are fibrant.
\end{Theorem}
\begin{proof}
By Theorem 4.1 of \cite{Rosicky2007On-combinatorial} if $\cc$ is combinatorial the full subcategory $W \hookrightarrow \Arr(\cc)$ is accessible and accessibly embedded.  Theorem~\ref{thm:Injectives} thus ensures that $W$ is a small injectivity class if and only if $W$ is closed under products in $\Arr(\cc)$ proving that $(1 \iff 2)$.  Arguing as in the proof of Corollary~\ref{cor:monad1}  the monad $T=UF$ is that induced by the accessible monadic $U:\Ainj(J) \to \Arr(\cc)$ of Theorem~\ref{thm:AlgebraicInjectives}.  Thus  $(2 \implies 4)$ whilst $(4 \implies 3)$ is trivial.   Since the forgetful functor from the category of $T$-algebras to the base $\Arr(\cc)$ creates products we obtain $(3 \implies 1)$.

Finally we use the well known fact --- which follows from Ken Brown's lemma --- that products of weak equivalences between fibrant objects are again weak equivalences.
\end{proof}

\section{Cone injectivity and weak equivalences of simplicial sets}\label{section:Cones}

It is not the case that the weak homotopy equivalences of simplicial sets can be described using injectivity nor as the algebras for a monad.  Indeed both injectives and those objects admitting algebra structure for a given monad are closed under all small products, whereas:

\begin{Proposition}
Weak homotopy equivalences of simplicial sets are not closed under countable products.
\end{Proposition}
\begin{proof}
Consider the reflexive directed graph $A$ with objects the natural numbers and with a unique map $n \to m$ if $m=n$ or $m=n+1$.  This has a single path component.  The countable product $A^{\omega}$, on the other hand, has more than one path component.  Its objects are countable sequences $(x_{i})$ and there exists a (unique) map $(x_{i}) \to (y_{i})$ if for each $j$ either $y_{j}=x_{j}$ or $y_{j}=x_{j}+1$.  Accordingly in $A^{\omega}$ there exists no path from $(1,1,1,....)$ to $(1,2,3,.....)$.

The category of reflexive directed graphs is the presheaf category $[\Delta_{1}^{op},\Set]$ where $\Delta_{1} \hookrightarrow \Delta$ is the full subcategory containing $[0]$ and $[1]$.  Left Kan extension along the inclusion yields the skeleton functor $S:[\Delta_{1}^{op},\Set]\to [\Delta^{op},\Set]$.  Observe that $SA$ has the same underlying reflexive graph as $A$.  Now $A=\cup_{n \in \mathbb N}A_{n}$ is a directed union of reflexive graphs $A_{n}$ where $A_{0}=\Delta_{1}(-,0)$ and where $A_{n+1} = A_{n} \cup_{\Delta_{1}(-,0)} \Delta_{1}(-,1)$ is the pushout obtained by attaching the edge from $n$ to $n+1$.  As is standard we denote the representable $\Delta(-,n)$ by $\Delta^n$.  Since $S$ preserves colimits and sends representables to representables we obtain $SA=\cup_{n \in \mathbb N}SA_{n}$ where $SA_{0}=\Delta^0$ and $SA_{n+1} = SA_{n} \cup_{\Delta^0} \Delta^1$.  The pushout coprojection $j_{n}:SA_{n} \to SA_{n} \cup_{\Delta^0} \Delta^1=SA_{n+1}$ is a pushout of the trivial cofibration $\Delta^{0} \to \Delta^{1}$ and so a trivial cofibration itself.  Therefore the countable composite of the chain of maps $(j_{n})_{n \in \mathbb N}$ is a trivial cofibration $\Delta^{0} =SA_{0} \to SA$.  It follows, by three from two, that the unique map $!:SA \to \Delta^0$ is a weak equivalence.  

On the other hand the countable product $!^{\omega}:(SA)^{\omega} \to (\Delta^0)^\omega \cong \Delta^0$ cannot be a weak equivalence --- for $\Pi_{0}((SA)^{\omega})$ is the set of path components of the underlying reflexive graph $A^{\omega}$ of $(SA)^{\omega}$, and this, as we have seen, has cardinality greater than $1$.
\end{proof}

In order to capture the weak equivalences of simplicial sets as injectives we pass from injectivity with respect to a set of morphisms to injectivity with respect to a set of \emph{cones}.  Cone injectivity was introduced by John \cite{John1977A-note} whilst a good textbook reference is \cite{Adamek1994Locally}.  

To motivate the general definition let us consider what it means for a morphism $f:X \to Y$ of simplicial sets to induce a surjection $\Pi_{0}f:\Pi_{0}X \to \Pi_{0}Y$ between sets of path components.  This amounts to asking that for each $y \in Y_{0}$ there exists $x \in X_{0}$ and a zigzag of $1$-simplices as below $$\cd{fx=x_0 \ar[r] & x_1 & x_2 \ar[l] \ar[r] & x_3 & \ldots \ar[l] \ar[r] & x_{2n-1} & x_{2n}=y \ar[l]}$$
where $n \in \mathbb N$.  (In the case that $Y$ is a Kan complex it suffices to take the case $k=1$ but in general we require all possible lengths.)

Let $Z_n$ denote the generic simplicial set containing a zigzag of 1-cells
$$\cd{0 \ar[r] & 1 & 2 \ar[l] \ar[r] & 3 & \ldots \ar[l] \ar[r] & 2n-1 & 2n \ar[l]}$$
and $j_0,j_{2n}:\Delta^0 \rightrightarrows Z_n$ the two maps selecting the endpoints.  We obtain a morphism
\begin{equation}\label{eq:cone1}
\cd{\varnothing \ar[d]_{!_{0}} \ar[r]^{!_{0}} & \Delta^0 \ar[d]^{j_{2n}}\\
\Delta^0 \ar[r]^{j_0} & Z_n} 
\end{equation}
in $\Arr(\SSet)$ and so a countable set of morphisms $$\{(!_{0},j_0):!_0 \to j_{2n}, n \in \mathbb N\}$$ with common source --- that is, a \emph{cone}.  We now see that $\Pi_{0}f$ is surjective exactly when each $(r,s):!_0 \to f$ factors through \emph{some} member $(!_0,j_0):!_0 \to j_{2n}$ of the cone.

Let us now turn to the general concept.  A cone $$p=\{p_{i}:A \to B_{i}:i \in I\}$$  in a category $\cc$ consists of a set of morphisms in $\cc$ with common source.  Given an object $X$ of $\cc$ we write $p \perp X$ if for each $f:A \to X$ 
there exists $i \in I$ and an extension
\begin{equation*}
\cd{A \ar[d]_{p_{i}} \ar[r]^{f} & X . \\
B_{i} \ar[ur]_{\exists} }
\end{equation*}
For a class of cones $J$ we write $J \perp X$ if $p \perp X$ for each $p \in J$.  We call $X$ injective and write $\Inj(J)$ for the full subcategory of $\cc$ consisting of the $J$-injectives.  Of course injectivity with respect to cones specialises to ordinary injectivity on considering cones containing a single arrow.

\subsection{Algebraic cone injectives}\label{sect:AlgebraicCones}
Let $J$ be a class of cones in a category $\cc$. An algebraic injective consists of a pair $(C,c_1,c_2)$ where, to begin with, we have $C \in \cc$.  Given a cone $p=\{p_{i}:A \to B_{i}:i \in I\} \in J$ and a morphism $f:A \to C$ we are provided with an index $c_{1}(p,f) \in I$ together with an extension of $f$
$$\cd{A \ar[d]_{p_{c_{1}(p,f)}} \ar[r]^{f} & C\\
B_{c_{1}(p,f)} \ar@{.>}[ur]_{c_{2}(p,f)}}$$
through the member of the cone indexed by $c_1(p,f)$.  Morphisms $g:(C,c_1,c_2) \to (D,d_1,d_2)$ respect both the choice of index and of extension.

In the case that the cones are just single morphisms this agrees with the notion of algebraic injective of Section~\ref{section:AInj}.  Observe that whilst an object $C$ is injective just when for each cone $p=\{p_{i}:A \to B_{i}:i \in I\} \in J$ the function 
\begin{equation}
\cd{\Sigma_{i \in I}\cc(B_{i},C) \ar[r]^{} & \cc(A,C)}
\end{equation}
is surjective, algebraically injectivity in the cone context enhances this by specifying a choice of section $(c_1(p,-),c_2(p,-))$ for each such function.  Under this viewpoint, the morphisms of algebraically injective objects are those commuting with the sections.  

As in \eqref{eq:pullback1} the above is concisely encoded by the fact that the square
\begin{equation}\label{eq:pullback2}
\cd{\Ainj(J) \ar[d]_{U} \ar[r]^{} & \SE([J,\Set]) \ar[d]^{V} \\
\f C \ar[r]^-{K} & \Arr([J,\Set])}
\end{equation}
is a pullback.  Here $\SE([J,\Set])$ is the category of \emph{split epimorphisms} in $[J,\Set]$ as before whilst this time $K$ sends $C$ to the family $(\Sigma_{i \in I}\cc(B_{i},C) \to \cc(A,C))_{p \in J}$.

For ordinary injectivity, we observed that if $\cc$ is locally presentable and $J$ a set of morphisms then $\Ainj(J)$ is locally presentable and $U:\Ainj(J) \to \cc$ an accessible monadic right adjoint.  In the setting of cones there is an analogous result.  

\begin{Proposition}\label{prop:AlgebraicConeInjectives}
Let $J$ be a set of cones in a locally presentable category $\cc$.  Then $\Ainj(J)$ is locally multi-presentable and $U:\Ainj(J) \to \cc$ is an accessible strictly monadic right multi-adjoint.
\end{Proposition}

Some readers may be unfamiliar with the \emph{multi}-aspects above.  We say enough about them only to prove the result.  The concepts of locally multi-presentable category and of multimonad were developed by Yves Diers \cite{Diers1980Categories,Diers1980Multi, Diers1981Some}.  Section 4 of \cite{Adamek1994Locally} is a useful textbook reference.

\begin{itemize}
\item A category is \emph{locally multi-presentable} just when it is accessible and has \emph{connected limits} (or, equivalently, \emph{multicolimits}). See \cite{Diers1980Categories, Adamek1994Locally}.
\item  A functor $U:\ca \to \cb$ is a \emph{right multiadjoint} if for each $B \in \cb$ there exists a cone $\eta=\{\eta_{i}:B \to UA_i:i \in I\}$ with the universal property that given $f:B \to UC$ there exists a unique pair $(i \in I,g:A_{i} \to C)$ such that $Ug \circ \eta_i = f$.  Note that the cone is determined up to unique isomorphism.
\item If $U:\ca \to \cb$ is a right multiadjoint one obtains a \emph{multi-monad} $T$ on $\cb$.  As in the classical setting this has a category of algebras $U^{T}:\TAlg \to \cb$ over $\cb$ and there is a canonical comparison $K:\ca \to \TAlg$ commuting with the forgetful functors to $\cb$.  As usual one says that $U$ is strictly monadic/monadic if $K$ is an isomorphism/equivalence.  
\end{itemize}

\begin{proof}[Proof of Proposition~\ref{prop:AlgebraicConeInjectives}]
Let $\cs$ be the free split epimorphism --- the category presented by the graph $\langle e:0 \leftrightarrows 1: m \rangle$ subject to the relation $e \circ m = 1$ --- and let $j:\atwo \to \cs$ be the identity on objects functor selecting the split epi $e$. Then the forgetful functor $V$ of \eqref{eq:pullback1} is $[j,1]:[\cs,[J,\Set]] \to [\atwo,[J,\Set]]$.  $V$ has a left adjoint $F$ given by left Kan extension along $j$.  Since $j$ is identity on objects $V=[j,1]$ strictly creates colimits and so is strictly monadic.

Next we show that $K$ preserves connected limits and is accessible.  Since (co)limits in $\Arr([J,\Set])$ are pointwise the functor $K$ preserves any (co)limits preserved by each of $$\Sigma_{i \in I}\cc(B_{i},-),\cc(A,-):\cc \to \Set$$ for $p=\{p_{i}:A \to B_{i}:i \in I\} \in J$.  Since $J$ is a set there exists a regular cardinal $\lambda$ such that the objects $A,B_{i}$ appearing in each cone $p$ are $\lambda$-presentable.  Since coproducts commute with colimits both of the above functors preserve $\lambda$-filtered colimits, whence so does $K$.

Now since $V$ has the isomorphism lifting property and both $V$ and $K$ preserve connected limits  it follows easily that the pullback $\Ainj(J)$ has such limits and colimits preserved by the pullback projections --- in particular preserved by $U$.  The isomorphism lifting property ensures that the square is a bipullback \cite{Joyal1993Pullbacks} and therefore, by Theorem 5.1.6 of \cite{Makkai1989Accessible}, the pullback $\Ainj(J)$ is accessible and the pullback projections accessible functors.

By a straightforward modification of the general adjoint functor theorem a functor between categories with connected limits has a right multiadjoint just when it satisfies the solution set condition and preserves connected limits (see \cite{Diers1981Some}).  By Proposition 6.1.2 of \cite{Makkai1989Accessible} each accessible functor satisfies the solution set condition.  It remains therefore to establish strict monadicity.  As a pullback of the strictly monadic $V$ the functor $U$ creates $U$-split coequalisers and so is strictly monadic by (the strict version of) Theorem 3.1 of \cite{Diers1980Multi}.\begin{footnote}{Theorem 3.1 of \cite{Diers1980Multi} concerns non-strict monadicity. The strict variant used here is a routine modification of its non-strict counterpart, just as for ordinary monads.}\end{footnote}
\end{proof}

Using the above result, we can give a novel proof that a small cone injectivity class in a locally presentable category is accessible and accessibly embedded, proceeding in much the same way as in the proof of Theorem~\ref{thm:Injectives}.  The full result, which appears as Theorem 4.17 of \cite{Adamek1994Locally}, is recorded below.

\begin{Theorem}\label{thm:ConeInjectives}\textnormal{(Ad{\'a}mek and Rosick{\'y})}
Let $\cc$ be locally presentable.  A full subcategory $j:\ca \hookrightarrow \cc$ is of the form $\Inj(J)$ for $J$ a set of cones if and only if $\ca$ is accessible and accessibly embedded.
\end{Theorem}
By Theorem 4.1 of \cite{Rosicky2007On-combinatorial} if $\cc$ is combinatorial the full subcategory $W \hookrightarrow \Arr(\cc)$ is always accessible and accessibly embedded.  Combining this fact with the preceding result and Proposition~\ref{prop:AlgebraicConeInjectives} we obtain:

\begin{Theorem}\label{thm:CombCone}
Let $\cc$ be a combinatorial model category with class of weak equivalences $W$. Then
\begin{enumerate}
\item $W \hookrightarrow \Arr(\cc)$ is of the form $\Inj(J)$ for $J$ a set of cones;
\item There exists a multimonad $T$ on $\Arr(\cc)$ such that $f$ admits $T$-algebra structure if and only if $f$ is a weak equivalence.
\end{enumerate}
\end{Theorem}

%
%


\subsection{Weak equivalences of simplicial sets as cone injectives}\label{section:WECones}

By Theorem~\ref{thm:CombCone} we know that the weak equivalences of simplicial sets form a cone injectivity class.  We will now, in fact, describe a countable set of cones generating the weak equivalences and extending the single cone of \eqref{eq:cone1} capturing surjectivity on $\Pi_0$.

Let $j_{n}:\partial \Delta^n \to \Delta^{n}$ be the inclusion of the boundary of the $n$-simplex.  We denote by $RH_n$ the pushout in 
\begin{equation}\label{eq:relhomotopy}
\cd{\partial \Delta^n \times \Delta^{1} \ar[d]_{j_{n} \times 1} \ar[r]^-{p_1} & \partial \Delta^n \ar[d] \\
\Delta^{n} \times \Delta^{1} \ar[r]^{} & \RH_{n}}
\end{equation}
since it classifies, by construction, homotopy relative to $\partial \Delta^n \to \Delta^{n}$.
Composing the isomorphism $\Delta^n \cong \Delta^n \times \Delta^0$ with the two maps $\Delta^0 \rightrightarrows \Delta^1$ corresponding to the $0$-simplices $0$ and $1$ of $\Delta^1$ produces a pair of maps $\Delta^n \rightrightarrows \Delta^n \times \Delta^1$.  Postcomposing these in turn with the morphism $\Delta^n \times \Delta^1 \to \RH_n$ of \eqref{eq:relhomotopy} produces a pair of maps $l_n,r_n:\Delta^{n} \rightrightarrows RH_{n}$ such that the square
$$
\cd{\partial \Delta^n \ar[d]_{j_{n}} \ar[r]^{j_{n}} & \Delta^{n} \ar[d]^{l_{n}}\\
\Delta^{n} \ar[r]_{r_{n}} & \RH_{n}}
$$
commutes. By Proposition 4.1 of \cite{Dugger2004Weak} a morphism $f:X \to Y$ of Kan complexes is a weak equivalence precisely when it is injective with respect to the set of morphisms $$\{\alpha_{n}=(j_{n},r_{n}):j_{n} \to l_{n}:n \in \mathbb N\}$$ in $\Arr(\SSet)$.

In constructing our generating cones Kan's fibrant replacement functor $Ex_{\infty}$ \cite{Kan1957On} plays an important role.  We will require an understanding of its construction and recall the relevant details now --- for more see \cite{Kan1957On, Goerss1999Simplicial}.  Non-degenerate $m$-simplices of $\Delta^n$ are in bijection with $(m+1)$-element subsets of $\{0,1,\ldots,n\}$ --- accordingly the set of non-degenerate simplices of $\Delta^n$ forms a poset, ordered by inclusion, whose nerve is by definition its subdivision $Sd \Delta^n$.  This construction extends to a functor $Sd:\Delta \to [\Delta^{op},\Set]$ which, by the Kan construction, extends along the Yoneda embedding to the left adjoint of an adjoint pair $Sd \dashv Ex:[\Delta^{op},\Set] \leftrightarrows [\Delta^{op},\Set]$.  The subdivision functor comes equipped with a natural map $p:Sd \to 1$ which, by adjointness, corresponds to a natural map $q:1 \to Ex$.  We write $Sd_n$ for the $n$-fold composite of $Sd$ and for $n>m$ with $p_{n,m}:Sd_n \to Sd_m$ denoting the composite of $p$-components; similarly $Ex_n$ and $q_{m,n}:Ex_m \to Ex_n$.

$Ex_{\infty}$ is defined as the colimit of the chain
\begin{equation*}
\cd{ 1 \ar[r]^{q_{0,1}} & Ex_{1} \ar[r]^{q_{1,2}} & Ex_{2} \ar[r]^{q_{2,3}} & Ex_3 \ar[r] & \ldots \ar[r] & Ex_{\infty}.}
\end{equation*}

As a fibrant replacement it has the property that a morphism $f:X \to Y$ is a weak equivalence just when $Ex_{\infty}f$ is a weak equivalence: that is, when $\{\alpha_{n}\}_{n \in \mathbb N} \perp Ex_{\infty}f$.  Now since $j_{n}$ is finitely presentable in $\Arr(\SSet)$ each morphism $j_{n} \to Ex_{\infty}f$ factors through a stage $Ex_{m}f$.  Using that $l_{n}$ is also finitely presentable we see that $\alpha_{n} \perp f$ if and only if for all $m \in \mathbb N$ and  $u:j_{n} \to Ex_{m}f$ there exists $k \geq m$ and  $u^\prime:j_{n} \to Ex_{k}f$ rendering commutative the square on the left below.
\begin{equation*}
\cd{
j_n \ar[d]_{\alpha_{n}} \ar[r]^-{u} & Ex_{m}f \ar[d]^{q_{m,k}}  &&& Sd_{k}j_n \ar[d]_{Sd_k\alpha_{n}} \ar[r]^-{p_{k,m}} & Sd_{m}j_n \ar[d]^{v} \\
l_n \ar@{.>}[r]^-{\exists u^{\prime}} & Ex_{k}f &&& Sd_{k}l_n \ar@{.>}[r]^{\exists v^{\prime}} & f}
\end{equation*}
By adjointness this is equally to say that for all $v:Sd_{m}j_{n} \to f$ there exists $k \geq m$ and a map $v^{\prime}:Sd_{k}l_{n} \to f$ such that the square above right commutes.  
Such a $v^{\prime}$ amounts to an extension of $v$ along the right vertical arrow in the pushout square below.
$$\cd{
Sd_{k}j_n \ar[d]_{Sd_k\alpha_{n}} \ar[r]^{p_{k,m}} & Sd_{m}j_n \ar[d]\\
Sd_{k}l_n \ar[r] & P_{m,n,k}}$$
Accordingly we see that $f$ is a weak equivalence just when for each pair $n,m \in \mathbb N$ $f$ is injective with respective to the cone
$$C_{n,m}=\{ Sd_mj_n \to P_{m,n,k}: k \geq m\}.$$
To describe this cone in more detail consider the following cube in which the top and bottom faces are pushouts.  The morphism $Sd_mj_n \to P_{m,n,k}$ of $C_{n,m}$ is given by the rightmost face, moving in the direction of the dotted arrows.
$$
\cd{
&& Sd_{m}\partial \Delta^n \ar[ddd] \ar@{.>}[dr] \\
Sd_{k}\partial \Delta^n \ar[ddd] \ar[dr] \ar[urr] &&& P^1_{n,m,k} \ar[ddd]\\
& Sd_{k}\Delta^n \ar[ddd] \ar[urr]\\
&& Sd_{m}\Delta^n \ar@{.>}[dr] \\
Sd_{k}\Delta^n \ar[dr] \ar[urr] &&& P^2_{n,m,k}\\
& Sd_{k}RH_n \ar[urr]}
$$

For a low dimensional example let $n=0$.  Now $\Delta^0$ and $\partial \Delta^0=\varnothing$ are fixed by $Sd$ whilst $\RH_0=\Delta^1$.  It follows that the right moving arrows on the back face of the cube are isomorphisms and, since the pushout of an isomorphism is an isomorphism, that the right face of the cube coincides with the left face --- in this case the square
$$
\cd{\varnothing \ar[d] \ar[r]^{} & \Delta^{0} \ar[d]\\
\Delta^{0} \ar[r] & Sd_{k}\Delta^1.}
$$   In fact, it is straightforward to show that $Sd_{k}\Delta^1$ is the generic zigzag
$$\cd{0 \ar[r] & 1 & 2 \ar[l] \ar[r] & 3 & \ldots \ar[l] \ar[r] & 2k-1 & 2k \ar[l]}$$
 of length $2k$ with the two maps $\Delta^0 \rightrightarrows Sd_{k}\Delta^1$ selecting the endpoints.  In particular $C_{0,0}$ is the cone \eqref{eq:cone1}. 

\section{From cone injectives and multimonads to injectives and monads}\label{section:Nikolaus}

We have seen that in order to describe the algebraic structure admitted by weak equivalences in a general combinatorial model category we must pass from the standard concepts of injectivity and monads to cone-injectivity and multi-monads.  On the other hand we now show that each combinatorial model category is Quillen equivalent to one in which all objects are fibrant --- in particular, in which the standard concepts suffice to capture the algebraic structure at hand.

Let $\cc$ be a combinatorial model with generating sets $I$ and $J$ of cofibrations and trivial cofibrations.  We define $\AFib = \Ainj(J)$ and, using the terminology of Nikolaus \cite{Nikolaus2011Algebraic}, refer to it the category of \emph{algebraically fibrant objects}.  In this case Theorem~\ref{thm:AlgebraicInjectives} ensures that $\AFib$ is locally presentable and $U:\AFib \to \cc$ a strictly monadic right adjoint.  

The two classes $(U^{-1}\WE,U^{-1}\Fib)$ in $\AFib$ consisting of the preimages of the weak equivalences and fibrations in $\cc$ specify the data for a Quillen model structure on $\AFib$ which, when the model category axioms are satisfied, we refer to as the projective model structure on $\AFib$.  

The first part of the following result modifies Theorem 2.20 of \cite{Nikolaus2011Algebraic}.  Although we require $\cc$ to be combinatorial rather than just cofibrantly generated, our result does not require the generating trivial cofibrations to be monomorphisms.  

The interesting feature of our argument, which is quite different to that of \emph{ibid.}, is that it involves the construction of a highly non-functorial path object.



\begin{Theorem}\label{thm:NikolausExtension}
Let $\cc$ be a combinatorial model category.  
\begin{enumerate}
\item The projective model structure on $\AFib$ exists, and is a combinatorial model structure with all objects fibrant.  The adjunction $F \dashv U:\AFib \leftrightarrows \cc$ is a Quillen equivalence.
\item The weak equivalences of algebraically fibrant objects form a small injectivity class.  In particular, there exists a monad $T$ on $\AFib$ such that $f:(A,a) \to (B,b)$ is a weak equivalence if and only if it bears $T$-algebra structure.
\end{enumerate}
\end{Theorem}

\begin{proof}
For (1) we start by observing that, by Theorem~\ref{thm:AlgebraicInjectives}, $\AFib$ is locally presentable.  Hence the sets $(FI,FJ)$ cofibrantly generate weak factorisation systems on $\AFib$ whose right classes are respectively $U^{-1}\Fib$ and $U^{-1}(\WE \cap \Fib)$.  Let $(C,c) \in \AFib$ and 
\begin{equation*}
\cd{
C \ar[r]^-{p} & PC \ar[r]^-{\langle s,t \rangle} & C^{2}
}
\end{equation*}
be a \emph{path object} factorisation in $\cc$: that is, a factorisation of the diagonal $\Delta:C \to C^{2}$ with $p$ a weak equivalence and $\langle s,t \rangle$ a fibration.  We will show that $PC$ can be equipped with the structure of an algebraically fibrant object $(PC,\phi)$ such that $p$ and $q$ lift to morphisms $p:(C,c) \to (PC,\phi)$ and $\langle s,t \rangle:(PC,\phi) \to (C,c)^{2}$ of algebraically fibrant objects.  By the dual of Proposition 2.2.1 of \cite{Hess2017A-necessary} --- a slight refinement of Quillen's path object argument --- the model structure will then exist.

The lifting function $\phi$ for $PC$ is defined in two stages.  Firstly, observe that since $\Delta:C \to PC$ is monic and $\Delta = \langle s, t \rangle \circ p$ we have that $p$ is monic too.  Now given a lifting problem $(j:A \to B \in J,f:A \to PC)$ suppose that $f$ factors through $p$ as $f^{\prime}:A \to C$ --- by monicity of $p$ the factorisation $f^{\prime}$ is unique.
\begin{equation*}
\cd{
A \ar[d]_{j} \ar@/^1.5pc/[rrr]^{f} \ar[rr]^{f^{\prime}} && C \ar[r]^-{p} & PC \\
B \ar[urr]|{c(j,f^{\prime})} \ar@/_1pc/[urrr]_{\phi(j,f)}
}
\end{equation*}
We then have the filler $c(j,f^{\prime})$ and define $\phi(j,f) = p \circ c(j,f^{\prime})$ as depicted above.  This definition ensures that $p:(C,c) \to (PC,\phi)$ is guaranteed to be a morphism of $\AFib$ independent of how we complete the definition of $\phi$.

If $f$ does not factor through $p$ we consider the composite $\langle s \circ f, t\circ f \rangle:A \to C^{2}$ and now use its lifting function to obtain an extension along $j$ as in the bottom horizontal arrow below. 
\begin{equation}\label{eq:lifting}
\cd{
A \ar[d]_{j} \ar[rrr]^{f} &&& PC \ar[d]^{\langle s,t \rangle} \\
B \ar[urrr]^{\phi(j,f)} \ar[rrr]_{\langle c(j,sf), c(j,tf)  \rangle } &&& C^{2}
}
\end{equation}
Then since $\langle s, t \rangle$ is a fibration there exists a diagonal filler and this defines $c(j,sf)$.  

We must prove that for general $f$ the equality
\begin{equation}\label{eq:preservation}
\langle s, t \rangle \circ \phi(j,f) = \langle c(j,s\circ f), c(j,t\circ f)  \rangle
\end{equation}
holds.  For $f$ not factoring through $p$ this is by construction.  If $f = p \circ f^{\prime}$ the left hand side of \eqref{eq:preservation} becomes $$\langle s, t \rangle \circ \phi(j,f) = \langle s, t \rangle \circ p \circ c(j,f^{\prime}) = \Delta \circ c(j,f^{\prime}) = \langle c(j,f^{\prime}), c(j,f^{\prime}) \rangle.$$  On the other hand the right hand side of \eqref{eq:preservation} becomes $$\langle c(j,s\circ f), c(j,t\circ f) \rangle = \langle c(j,s\circ p\circ f^{\prime}), c(j,t\circ p \circ f^{\prime}) \rangle = \langle c(j,f^{\prime}), c(j,f^{\prime}) \rangle$$
as required, where the last step uses that $s\circ p=1$ and $t\circ p=1$.   

Accordingly we obtain the model structure and, since $U$ preserves fibrations and weak equivalences, the adjunction is a Quillen adjunction. Moreover, since $U$ reflects fibrations and each object in its image is fibrant, it follows that all objects in $\AFib$ are fibrant. 

Let us show that the unit component $\eta_{A}:A \to UFA$ belongs to $\llp(\rlp{J})$.  This follows directly from Garner's work on algebraic weak factorisation systems \cite{Garner2011Understanding} on observing that $UF$ is the \emph{fibrant replacement monad} associated to the algebraic weak factorisation system on $\cc$ freely generated by the inclusion $J \to \Arr(\cc)$.  For completeness we give a short elementary argument.  Consider a lifting problem as in the outside of the diagram below.
\begin{equation*}
\cd{
A \ar@{.>}[r]^{k} \ar[d]_{\eta_{A}} \ar@/^1.5pc/[rr]^{r} & P \ar[dl]^-{q \in \rlp{J}} \ar[r]^{p} & X \ar[d]^{f \in \rlp{J}} \\
UFA \ar[rr]_{s} && Y
}
\end{equation*}
Since $f \in \rlp{J}$ so is its pullback $q$; combining its lifting property with that of $UFA$ (in the style of \eqref{eq:lifting} above) we can equip $P$ with the structure of an algebraically fibrant object $(P,p)$ such that $q:(P,p) \to FA \in \AFib$.  The map to the pullback $k:A \to P=U(P,p)$ then induces a unique morphism $l:FA \to (P,p) \in \AFib$ with $l \circ \eta_{A} = k$.  Since $q:(P,p) \to FA \in \AFib$ the universal property also ensures that $q \circ l = 1$.  The composite diagonal $p \circ l:UFA \to P \to X$ then gives the desired filler.  Therefore the unit of the adjunction $\eta_{A}:A \to UFA$ belongs to $\llp{(\rlp{J})}$  and so is a weak equivalence.  Since $U\epsilon_{(C,c)} \circ \eta_{U(C,c)} = 1$ three for two ensures that $U\epsilon_{(C,c)}$ is a weak equivalence in $\cc$.  Therefore $\epsilon_{(C,c)}$ is a weak equivalence in $\AFib$ and the adjunction a Quillen equivalence.

It remains to prove (2).  Since each object of $\AFib$ is fibrant in the projective model structure and since the model structure is combinatorial, this follows immediately from Theorem~\ref{thm:CombinatorialInj}.
\end{proof}

\appendix

\section{Free algebras for pointed endofunctors}\label{sect:Pointed}

In order to make the proof of Theorem~\ref{thm:AlgebraicInjectives} accessible to a broader audience we now describe in detail the construction of free algebras for pointed endofunctors.  The classical reference is \cite{Kelly1980A-unified}, specifically Theorems~14.3 and 15.6.  Here we take a different approach to essentially the same result.  Our approach is based upon, and is a straightforward modification of, Koubek and Reiterman's elegant construction of the free algebra on an endofunctor \cite{Koubek1979Categorical}.  One of the attractive features of this approach is that it emphasises the explicit formulae involved --- see Proposition~\ref{prop:free} below --- by focusing not only on the free algebra but also on the \emph{free algebraic chain}.

To begin with, a chain is a functor $X:\Ord \to \C$ on the posetal category of ordinals, whilst a chain map is a natural transformation.   Given a pointed endofunctor $(T,\eta)$ on $\C$ an algebraic chain $(X,x)$ is a chain $X$ together with, for each ordinal $n$, a map $x_n:TX_n \to X_{n+1}$ satisfying
\begin{itemize}
\item for all $n$
\begin{equation}\label{eq:unit}
\xymatrix{
X_{n} \ar[drr]_{j_{n}^{n+1}} \ar[rr]^{\eta_{X_{n}}} && TX_{n} \ar[d]^{x_{n}} \\
 && X_{n+1}
}
\end{equation}
\item and for all $n < m$ the diagram 
\begin{equation}\label{eq:alg}
\xymatrix{
TX_{n} \ar[d]_{x_{n}} \ar[rr]^{T(j_{n}^{m})} && TX_{m} \ar[d]^{x_{m}} \\
X_{n+1} \ar[rr]_{j_{n+1}^{m+1}} && X_{m+1}
}
\end{equation}
commutes.
\end{itemize}
A morphism $f:(X,x) \to (Y,y)$ of algebraic chains is a chain map that commutes with the $x_n$ and $y_n$ for all $n$.  These are the morphisms of the category  $\TAlg_\infty$ of algebraic chains.

\begin{example}
Let $J$ be a set of morphisms in $\C$.  In Section ~\ref{section:AInj} we described the pointed endofunctor $(R,\eta)$ whose algebras are algebraic injectives.  Using the construction of $R$ in \eqref{eq:rconstruction} we see that an algebraic chain is a chain $X$ together with, for each lifting problem $(\alpha:A \to B \in J, f:A \to X_{n})$, a filler $x_{n}(\alpha,f)$ rendering the left square below commutative.
\begin{equation*}
\xymatrix{
A \ar[d]_{\alpha} \ar[rr]^{f} && X_{n}\ar[d]^{j_n^{n+1}} \ar[rr]^{j_n^m} && X_m \ar[d]^{j_m^{m+1}} \\
B \ar@/_1.2pc/[rrrr]_{x_{m}(\alpha,j_{n}^{m} \circ f)} \ar[rr]^{x_{n}(\alpha,f)} && X_{n+1} \ar[rr]^{j_{n+1}^{m+1}} && X_{m+1}
}
\end{equation*}
These fillers must satisfy the indicated compatibility for $n<m$.
\end{example}

There is a forgetful functor $V:\TAlg_\infty \to \C$ sending $(X,x)$ to $X_0$.  Our first goal is to show that if $\C$ is cocomplete then $V$ has a left adjoint.  

To this end we first observe that the equation \eqref{eq:alg} holds for all $n < m$ if it does so in the cases (a) $m=n+1$ and (b) $m$ is a limit ordinal.  Now consider a chain $X$ equipped with maps $x_{n}:TX_n \to X_{n+1}$ satisfying \eqref{eq:unit}.  Then case (a) of \eqref{eq:alg} becomes the assertion that for all $n$ the diagram
\begin{equation}
\xymatrix{
TX_{n} \ar@<.6ex>[rr]^{Tx_{n} \circ T\eta_{X_{n}}}\ar@<-.6ex>[rr]_{Tx_{n} \circ \eta_{TX_{n}}} && TX_{n+1} \ar[r]^{x_{n+1}} & X_{n+2}
}
\end{equation}
is a fork.  Case (b) of \eqref{eq:alg} asserts that for all limit ordinals $m$ and $n<m$ the diagram
\begin{equation*}
\xymatrix{
TX_{n} \ar@<0.5ex>[rr]^{ T{j_n^m}} \ar@<-0.5ex>[rr]_{ \eta_{X_{m}} \circ j_{n+1}^{m} \circ x_{n}} && TX_{m} \ar[r]^{x_{m}} & X_{m+1}
}
\end{equation*}
is a fork.  To see this, use that $x_{m} \circ \eta_{X_{m}} = j_{m}^{m+1}$.  In the presence of filtered colimits this equally asserts that for each limit ordinal $m$ the diagram 
\begin{equation}
\xymatrix{
col_{n<m}TX_{n} \ar@<0.5ex>[rr]^{\langle T{j_n^m} \rangle} \ar@<-0.5ex>[rr]_{\langle \eta_{X_{m}} \circ j_{n+1}^{m} \circ x_{n} \rangle} && TX_{m} \ar[r]^{x_{m}} & X_{m+1}
}
\end{equation}
is a fork.

\begin{Proposition}\label{prop:free}
If $\C$ is cocomplete then $V$ has a left adjoint whose value at $X \in \C$ is the algebraic chain $X_{\bullet}$ with values:
\begin{itemize}
\item $X_{0}=X$, $X_{1}=TX$, $j_{0}^{1}=\eta_{X}:X \to TX$ and $x_0=1:TX \to TX$.
\item At an ordinal of the form $n+2$ the object $X_{n+2}$ is the coequaliser
\begin{equation*}
\xymatrix{
TX_{n} \ar@<.6ex>[rr]^{Tx_{n} \circ T\eta_{X_{n}}}\ar@<-.6ex>[rr]_{Tx_{n} \circ \eta_{TX_{n}}} && TX_{n+1} \ar[r]^{x_{n+1}} & X_{n+2}
}
\end{equation*}
with $j_{n+1}^{n+2}=x_{n+1} \circ \eta_{X_{n+1}}$.
\item At a limit ordinal $m$,
\begin{itemize} 
\item $X_{m} = col_{n < m} X_{n}$ with the connecting maps $j_n^m$ the colimit inclusions.  
\item $X_{m + 1}$ is the coequaliser
\begin{equation*}
\xymatrix{
col_{n<m}TX_{n} \ar@<0.5ex>[rr]^{\langle T{j_n^m} \rangle} \ar@<-0.5ex>[rr]_{\langle \eta_{X_{m}} \circ j_{n+1}^{m} \circ x_{n} \rangle} && TX_{m} \ar[r]^{x_{m}} & X_{m+1}
}
\end{equation*}
with $j_{m}^{m+1}=x_{m} \circ \eta_{X_{m}}$.
\end{itemize}
\end{itemize}
\end{Proposition}
\begin{proof}
The unit of the adjunction will be the identity --- so, we are to show that given $f:X \to Y_{0}=V(Y,y)$ there exists a unique map $f:X_\bullet \to (Y,y)$ of algebraic chains with $f_{0}=f$.  The required commutativity below left
\begin{equation*}
\xymatrix{
TX \ar[d]_{Tf} \ar[r]^{x_{0}=1} & TX \ar[d]^{f_1} \\
TY_0 \ar[r]^{y_0} & Y_{1}}
\hspace{1cm}
\xymatrix{
TX_{n} \ar[d]^{Tf_{n}} \ar@<.6ex>[rr]^{Tx_{n} \circ T\eta_{X_{n}}}\ar@<-.6ex>[rr]_{Tx_{n} \circ \eta_{TX_{n}}} && TX_{n+1} \ar[d]^{Tf_{n+1}} \ar@{.>}[r]^{x_{n+1}} & X_{n+2} \ar@{.>}[d]^{f_{n+2}} \\
TY_{n} \ar@<.6ex>[rr]^{Ty_{n} \circ T\eta_{Y_{n}}}\ar@<-.6ex>[rr]_{Ty_{n} \circ \eta_{TY_{n}}} && TY_{n+1} \ar[r]^{y_{n+1}} & Y_{n+2}
}
\end{equation*}
forces us to set $f_{1}=y_{0} \circ Tf$.  The map $f_{n+2}$ must render the right square in the diagram above right commutative.  But since the two back squares serially commute and the bottom row is a fork there exists a unique map from the coequaliser $X_{n+2}$ rendering the right square commutative.  This uniquely specifies $f_n$ for $n < \omega$.  At a limit ordinal $m$, $f_m:X_{m}=col_{n < m} X_{n} \to Y_{m}$ is the unique map from the colimit commuting with the connecting maps --- which it must do to form a morphism of chains.  At the successor of a limit ordinal $m$ there is a unique map $f_{m+1}:X_{m+1} \to Y_{m+1}$ from the coequaliser satisfying $f_{m+1} \circ x_{m} = y_{m} \circ Tf_{m}$, as required.
\end{proof}

The usual forgetful functor $U:\TAlg \to \C$ factors through $V:\TAlg_\infty \to \C$ via a functor $\Delta:\Alg \to \Alg_\infty$: this sends $(X,x)$ to the constant chain on $X$ equipped with $x_n=x$ for all $n$.  A chain $X$ is said to stabilise at an ordinal $n$ if for all $m>n$ the map $j_{n,m}:X_n \to X_m$ is invertible.  Observe that if an algebraic chain $(X,x)$ stabilises at $n$ then $X_{n}$ equipped with the $T$-algebra structure
\begin{equation}\label{eq:structure}
(j_n^{n+1})^{-1} \circ x_n:TX_n \to X_{n+1} \cong X_n
\end{equation} is a reflection of $(X,x)$ along $\Delta$.  In particular:
\begin{Proposition}
If $X_{\bullet}$ stabilises at $n$ then $X_{n}$, with structure map as in \eqref{eq:structure}, is the free $T$-algebra on $X$.
\end{Proposition}

Accordingly we examine circumstances under which each $X_{\bullet}$ stabilises.  In the following the term \emph{chain of length $n$} refers to a functor $X:\Ord_{<n} \to \C$ from the full subcategory of ordinals less than $n$. 

\begin{Proposition}\label{prop:stab}
If $T$ preserves the colimit $X_{m}=col_{n<m}X_{n}$ for $m$ a limit ordinal then $X_{\bullet}$ stabilises at the ordinal $m$.
\end{Proposition}
\begin{proof}
Firstly one shows that $j_m^{m+1}:X_{m} \to X_{m+1}$ is invertible.  To see this observe that the morphisms $x_{n}:TX_n \to X_{n+1}$ form a morphism of chains of length $m$, and so induce a map $x_m:TX_m \to X_m$ between the colimits.  This has the universal property of the coequaliser $x_m:TX_{m} \to X_{m+1}$ whereby the comparison $j_m^{m+1}$ between the two coequalisers is invertible.

Now the coequaliser formulae allow to us to prove that if for some $k$ the map $j_k^{k+1}$ is invertible then so is $j_{k+1}^{k+2}$ and, likewise, that if $j_k^l$ is invertible for $k < l$ with $l$ a limit ordinal then $j_l^{l+1}$ is invertible.  Given that $j_m^{m+1}$ is invertible it easily follows from these facts, using transfinite induction, that each $j_m^n$ is invertible for all $n>m$.
\end{proof}


\begin{Theorem}
Let $(T,\eta)$ be a pointed endofunctor on a cocomplete category $\C$.  If either
\begin{enumerate}
\item $T$ preserves colimits of $n$-chains for some limit ordinal $n$, or
\item $\C$ is equipped with a well copowered proper factorisation system $(\ce,\cm)$ such that $T$ preserves colimits of $\cm$-chains of length $n$ for some limit ordinal $n$.
\end{enumerate}
Then free $T$-algebras exist: namely, each algebraic chain $X_{\bullet}$ stabilises and its point of stabilisation, with algebra structure as in \eqref{eq:structure}, is the free $T$-algebra on $X$.
\end{Theorem}
\begin{proof}
Assuming (1) the conclusion holds on combining the three preceding propositions.  Assuming (2), it suffices to show that if $A$ is any chain, then there exists a limit ordinal $m$ such that $T$ preserves the colimit of the chain $(A_{n})_{n < m}$ of length $n$.  This is the content of a clever lemma from Section 8.5 of Koubek and Reiterman \cite{Koubek1979Categorical}.  See also Proposition 4.1 of \cite{Kelly1980A-unified} for a helpful proof of that result.
\end{proof}

\black

\end{document}